\documentclass[a4paper,11pt]{article}
\usepackage{amscd, amsmath, amssymb, amsfonts, epic, amsthm}
\usepackage{graphicx}
\usepackage{tikz}
\usetikzlibrary{patterns}
\usepackage{a4wide}
\usepackage{xypic}
\usepackage{paralist}
\usepackage{booktabs}
\usepackage{multirow}
\usepackage{epsfig}
\usepackage{graphics}

\usepackage{caption}
\usepackage{subcaption}

\title{Partially ample line bundles on toric varieties}
\author{Nathan Broomhead, John Christian Ottem, Artie Prendergast-Smith}

\date{}

\def\Z{\text{\bf Z}}
\def\Q{\text{\bf Q}}
\def\R{\text{\bf R}}

\def\kk{\text{\bf k}}
\def\P{\text{\bf P}}

\def\arrow{\rightarrow}

\def\iso{\cong}

\def\Pic{\text{Pic}} 

\def\O{\mathcal O}

\def\E{{\mathcal E}}

\newtheorem{theorem}{Theorem}[section]
\newtheorem{lemma}[theorem]{Lemma}
\newtheorem{corollary}[theorem]{Corollary} 
\newtheorem{proposition}[theorem]{Proposition}

\newtheorem{definition}[theorem]{Definition}

\begin{document}

\maketitle

Ample line bundles are a fundamental concept in algebraic geometry,
encapsulating the central notion of positivity. A natural extension of
the notion of ampleness is that of $q$-ampleness, for nonnegative
integers $q$. Roughly speaking, $q$-ample line bundles on a variety
are those which ``kill cohomology in degrees above $q$.'' Line bundles
of this kind have been studied by a number of authors, including
Andreotti--Grauert \cite{AG}, Sommese \cite{Sommese},
Demailly--Peternell--Schneider \cite{DPS}, and Totaro \cite{Totaro}.

In this note, we verify some basic properties of $q$-ample line
bundles on toric varieties. We begin by reviewing basic
facts about $q$-ampleness. Then in Section \ref{cone} we study the structure of the set of
all $q$-ample line bundles on a simplicial toric variety. We show that the cone
of $q$-ample line bundles is the interior of a finite union of
rational polyhedral cones, and that it is defined by the vanishing of
asymptotic cohomological functions. As an illustration, in Section
\ref{examples} we give examples of explicit calculations of $q$-ample cones of two families of
toric varieties.

 In Sections \ref{restriction} and \ref{flip} we prove that $q$-ampleness of big line
 bundles on toric varieties is detected by restriction to torus-invariant divisors, and
 use this fact to study $q$-ampleness of the anticanonical bundle: in
 particular, we give an example showing that 1-ampleness of $-K$ is
 not preserved by flips, answering a question of
 Demailly--Peternell--Schneider. We finish in Section \ref{kodaira} by
 proving a Kodaira-type vanishing theorem for $q$-ample bundles on toric varieties.

\section{$q$-ample line bundles}
Throughout the paper we stick to toric varieties over algebraically closed fields of characteristic zero. We switch between additive and multiplicative notation for line bundles wherever convenient, and freely identify line bundles and the corresponding divisors. 

In the 1950s Serre gave a cohomological characterisation of ample line
bundles: a line bundle is ample if and only if some sufficiently high
power of it kills cohomology of any coherent sheaf in degrees above
zero. This characterisation suggests the following generalisation of
ampleness, introduced by Sommese \cite{Sommese}. (Note that Sommese's
definition requires that some power of the line bundle be globally
generated, but we drop that hypothesis here.) 

\begin{definition} \label{def-naive}
Let $X$ be a projective variety. A line bundle $L$ on $X$ is called
{\it $q$-ample} (for some integer $q \geq 0$) if for any coherent
sheaf $F$ on $X$, there exists a natural number $n_0$ (depending on
$F$) such that
\begin{align} \label{vanishing}
H^i(X, L^n \otimes F) &=0 \text{ for all } i>q \text{ and } n \geq n_0.
\end{align}
\end{definition}
Any line bundle on a variety of dimension $n$ is $n$-ample; by Serre,
$0$-ample is the same as ample.

At first sight the $q$-ample condition seems hard to check, since it
involves tensoring with an arbitrary coherent
sheaf. The following result reduces this to a much simpler condition:

\begin{proposition}[\cite{Ottem}, Lemma 2.1] \label{q-ample-prop}
Let $X$ be a projective variety over a field of characteristic 0, and
fix an ample line bundle $\O(1)$ on $X$. A line bundle $L$ on $X$ is
$q$-ample if and only if for each $r \geq 0$, we have $H^i(X,L^m \otimes
\O(-r))=0$ for $m$ sufficiently large and all $i>q$. In particular,
condition \eqref{vanishing} need only be checked for locally free
sheaves.
\end{proposition}For $q>0$, it remains an open problem to give a simple numerical or geometric condition, in the spirit of Kleiman's criterion, for $q$-ampleness.

\section{The $q$-ample cone of a toric variety} \label{cone}

In this section we recall some basic facts about cohomology of line bundles on toric varieties, together with Hering--K\"uronya--Payne's formula for calculating dimensions of cohomology groups. We then use to describe the structure of the cone of $q$-ample line bundles on a toric variety, and to show that it is characterised by the vanishing of asymptotic cohomological functions.


Let $X = X(\Delta)$ be a simplicial projective $n$-dimensional toric
variety, corresponding to some complete fan $\Delta$ in a lattice $N
\cong \Z^n$.  We denote by $\Delta(1)$ the set of rays of $\Delta$, and write $v_i$  for the primitive generators of the ray $i \in \Delta(1)$. There is a one-to-one
correspondence between prime torus-invariant divisors and rays \cite[Chapter 3]{Fulton}. We denote these divisors by $\{D_i
\mid i \in \Delta(1)\}$ and the free group generated by them by $\Z^{\Delta(1)}$. The
dual space $\Z_{\Delta(1)}$ is generated by the dual basis $\{e_i \mid i \in
\Delta(1) \}$.

Let $ M := \operatorname{Hom}(N, \Z) \cong \Z^n$ be the
dual lattice to $N$, with pairing $ \langle \cdot , \cdot \rangle
$. We have the following commutative diagram with exact
rows:

\begin{equation} \label{eqn:seseq} \xymatrix{
0 \ar[r] & M  \ar@{=}[d]      \ar[r] & \mathrm{Div}_T(X)   \ar@{^{(}->}[d]    \ar[r] & \Pic(X) \ar@{^{(}->}[d] \ar[r] & 0 \\
0 \ar[r] & M  \ar[r] & \Z^{\Delta(1)}  \ar[r]^{[-]} & Cl(X) \ar[r] & 0 \\
} \end{equation}
where $\mathrm{Div}_T(X)$ is the group of torus-invariant Cartier
divisors, and $Cl(X)$ is the class group. Applying the functor
$-\otimes_{\Z} \R$ and using the fact that simplicial toric varieties
are $\Q$-factorial, we obtain the following exact sequence
\begin{equation} \label{eqn:seseq2} \xymatrix{
0 \ar[r] & M_\R  \ar[r] & \R^{\Delta(1)}  \ar[r]^{[-]} & N^1(X) \ar[r] & 0 \\
} \end{equation}
where $N^1(X)$ denotes $\Pic(X) \otimes \R$.

For $I \subset \Delta(1)$, we define $\Delta_I$ to be the subfan of $\Delta$
consisting of cones whose rays belong to $I$. For a torus-invariant
divisor $D=\sum_\rho a_\rho D_\rho$, we define the polyhedral region
\begin{align} \label{polytope}
P_{D,I} = \left\{ u \in M_{\mathbf R} \, | \, \left<u,v_\rho \right> \geq -a_\rho \Leftrightarrow \rho \in I \right\}.
\end{align}
Note that for a positive integer $m$, we have $P_{mD,I}=mP_{D,I}$. 

Hering--K\"uronya--Payne \cite{HKP} gave a description of the cohomology of the
divisor $D$ in terms of local cohomology groups and lattice points in
$P_{D,I}$.  For a subfan $\Delta_I$, we denote the dimension of the
(topological) local cohomology group $H^i_{|\Delta_I|}(N_\mathbf R)$
with support in the subspace $|\Delta_I|$ by $h^i_{|\Delta_I|}$.

\begin{theorem}[Hering--K\"uronya--Payne]
For a line bundle $D$ on a simplicial projective toric variety we have
\begin{align*}
h^i(X,D) &= \sum_{I \subset \Delta(1)} h^i_{|\Delta_I|} \cdot \# \left( P_{D,I} \cap M \right)
\end{align*}
\end{theorem}

We also recall K\"uronya's definition of {\it asymptotic cohomological
  functions}. For an $n$-dimensional projective variety $X$ and a line
bundle $L$ on $X$, we define

\begin{align*}
\widehat{h}^i (L) = n! \lim_{m \to \infty} \frac{h^i(X,mL)}{m^n}
\end{align*}
Note that $\widehat{h}^0$ is just the usual volume function. K\"uronya showed that the functions $\widehat{h}^i$ give well-defined homogeneous continuous functions on the space $N^1(X)$. By Serre it is clear that for each $i >0$ the function $h^i$ vanishes identically on the nef cone; de Fernex--K\"uronya--Lazarsfeld \cite{dFKL} showed that in fact this characterises the nef cone. 

For a bounded polyhedron $P \in M_\mathbf R$, let $\operatorname{Vol}( P)$ denote the volume of $P$, normalised so that the smallest lattice simplex has unit volume. Then
\begin{align*}
  \operatorname{Vol} (P) &= n! \lim_{m \to \infty} \frac{\# mP \cap M}{m^n}.
\end{align*}
Combining the three displayed equations above then gives
\begin{align} \label{formula-hkp}
\widehat{h}^i (D) = \sum_{I \subset \Delta(1)} h^i_{|\Delta_I|} \cdot \operatorname{Vol} (P_{D,I}).
\end{align}

Finally we need the following elementary result.

\begin{lemma} \label{perturb}
Let $A$ be a real $m \times n$ matrix with $m\le n$, and let $b \in \mathbf R^m$. If the inequality $Ax \leq b$ has a solution, then there exists a vector $v \in \mathbf R^m$ such that for any $\epsilon>0$ the set $\{ x \in \mathbf R^n \, |\, Ax \leq b + \epsilon v \}$ is an $n$-dimensional polyhedron. 
\end{lemma}

Now we can state the main result of this section.

\begin{theorem} \label{thm-main}
Let $X$ be a simplicial projective toric variety and let $D=\sum_\rho
a_\rho D_\rho$ be a torus-invariant divisor on $X$. Then the following
are equivalent:
\begin{enumerate}
\item[(a)] $D$ is $q$-ample;
\item[(b)] For every ample divisor $A$, we have $H^i(X,mD-A)=0$ for $i>q$ and $m \gg0$;
\item[(c)] There exists an open neighbourhood $U$ of $[D] \in N^1(X)$ such that $\widehat{h}^i(x)=0$ for all $i>q$ and all $x \in U$.
\end{enumerate}
\end{theorem}
\begin{proof}
The equivalence $(a) \Leftrightarrow (b)$ is Proposition \ref{q-ample-prop}.


Next, if $D$ is $q$-ample, then immediately from the definition we get $\widehat{h}^i(D)=0$ for all $i>q$. Moreover, for each $q$, the $q$-ample cone is open in $N^1(X)$ \cite[Theorem 8.3]{Totaro}. This proves the implication $(a) \Rightarrow (c)$. 

So it suffices to prove the implication $(c) \Rightarrow (b)$. Fix $i>q$.  Let $I \subset \Delta(1)$ be a subset such that $h^i_{|\Delta_I|} >0$: that is, a subset which could contribute nonzero terms to $\widehat{h}^i$. Then for any $D' \in U$, the polyhedron $P_{D',I}$ must be the empty set: if it were not, by Lemma \ref{perturb} we could choose a class $E$ such that for all $0<\epsilon\ll1$, the perturbed class $D'+\epsilon E \in U$ has the property that $P_{D'+\epsilon E, J}$ has positive volume. But then formula \eqref{formula-hkp} implies that $\widehat{h}^i(X,D'+\epsilon E)$ is nonzero, contradicting our choice. Since for large $m$ we know that $D-\frac{1}{m}A$ lies in $U$, we must have that $P_{D-\frac{1}{m}A,I} = \emptyset$. By our earlier remark, this implies that $P_{mD-A,I}=\emptyset$ too. Since this is true for all $I$ which could contribute to $H^i$, we get that $H^i(X,mD-A)=0$ as required. 
\end{proof}

The theorem allows us to describe the $q$-ample cone of a toric variety. By definition, an $\mathbf R$-divisor $D$ is $q$-ample if it is numerically equivalent to a sum $cL+A$ where $L$ is a $q$-ample divisor, $c$ a positive real number, and $A$ is an ample $R$-divisor. The set of all $q$-ample $\R$-divisors defines an open cone in $N^1(X)$ whose integer points are exactly the $q$-ample divisors \cite[Theorem 8.3]{Totaro}. Chen--Lazarsfeld asked if, for Fano varieties, these cone are always the interior of a finite union of rational polyhedral cones. (This is the simplest possible generalisation of the Cone Theorem in this context, since these cones are known not to be convex in general.) Here we give a positive answer to the analogue of Chen--Lazarsfeld's question for toric varieties.

\begin{corollary} \label{cones}
If $X$ is a simplicial projective toric variety, then the closure of the $q$-ample cone $\overline{Amp}_q(X)$ is a union of rational polyhedral cones, for each $q \geq 0$.
\end{corollary}
\begin{proof}
Theorem \ref{thm-main} says that $\overline{Amp}_q(X) \subset N^1(X)$ is the common vanishing locus of the functions $\widehat{h}^i$ for $i>q$. Formula \eqref{formula-hkp} shows that $\widehat{h}^i(D)=0$ for all $i>q$ if and only if:
\begin{align*}
\operatorname{Vol}P_{D,I}=0 \text{ for all } I \subset \Delta(1) \text { such that } h^i_{|\Delta|_I} >0 \text{ for some } i>q.
\end{align*}
The basic point is that for each appropriate $I$, the subset in $\mathbf R^{\Delta(1)}$ of $D$ satisfying the above condition is cut out by a collection of rational hyperplanes, and the images of these hyperplanes in $N^1(X)$ then define the cone $\overline{Amp}_q(X)$. 

In more detail, let $I$ be any subset of $\Delta(1)$, and consider a divisor $D=\sum d_\rho D_\rho$. First, if $d_\rho>0$ for $\rho \in I$ and $d_\rho <0$ for $\rho \notin I$, then clearly $P_{D,I}$ contains a small ball around the origin: in particular, $\operatorname{Vol}P_{D,I}>0$. Note that if we replace $D$ with a linearly equivalent divisor $D'=D+\operatorname{div}(u)$ (for some rational function $u \in M$), the new polytope $P_{D',I}$ is just the translate $P_{D,I}-U$, so it also has positive volume.

 Conversely, if $P_{D,I}$ contains a ball around the origin, it is
 clear we must have $d_\rho>0$ for $\rho \in I$ and $d_\rho <0$ for
 $\rho \notin I$. Now if $D$ is any divisor such that
 $\operatorname{Vol}P_{D,I}>0$, then (perhaps after scaling $D$) there
 exists a rational function $u$ such that for the divisor
 $D'=D+\operatorname{div(u)}$, the polytope $P_{D',I}$ contains a
 small ball around the origin, and so $D'$ has the property stated.

To summarise, we have shown that $\operatorname{Vol}P_{D,I}>0$ if and only if $D$ is linearly equivalent to a divisor $D'=\sum d_\rho D_\rho$ with $d_\rho>0$ if and only if $\rho \in I$. The set of such divisors $D'$ forms an (open) orthant $O_I$ in $\mathbf R^{\Delta(1)}$, and so the closure of its image $\overline{[O_I]}$ in $N^1(X)$ is a rational polyhedral cone. Theorem \ref{thm-main} shows that $Amp_q(X)$ is the complement of the union (over a finite set of $I$) of the cones $\overline{[O_I]}$, and hence is the interior of a union of rational polyhedral cones. 
\end{proof}

\section{Examples} \label{examples}
In this section, we illustrate Theorem \ref{thm-main} by calculating
the $q$-ample cones of two families of examples: projective bundles
over $\mathbf P^1$, and toric $\mathbf P^1$-bundles over projective
spaces.

To make the calculations easier, we find it convenient to reformulate our earlier conditions on cohomology vanishing in terms of the polytope of our toric variety $X$. Recall that the polytope $P_X$ of $X$ is defined exactly as in formula \eqref{polytope} in Section \ref{cone}, where $D$ is chosen to be any ample divisor on $X$, and $J=\Delta(1)$. Broomhead \cite{Broomhead} showed how to calculate cohomology of line bundles on $X$ in terms of the topology of certain subspaces of $P_X$. For this statement, given a subset $I \subset \Delta(1)$, let $P_X^I$ denote the subset of $P_X$ consisting of the union of all top-dimensional faces corresponding to rays in $I$. The statement we need is the following:

\begin{proposition} \label{prop-orthants}
Let $X$ be a simplicial projective toric variety and $D=\sum a_\rho D_\rho$ a torus-invariant divisor. Then $\widehat{h}^i(D) \neq 0$ if and only if the following is true: there exists $I \subset \Delta(1)$ such that $\widetilde{H}^{i-1}(P_X^I) \neq 0$, and a divisor $D'=\sum d_\rho D_\rho$, linearly equivalent to $D$, such that $d_\rho < 0$ if and only if $\rho \in I$.  
\end{proposition}
\begin{proof}
According to formula \eqref{formula-hkp}, $\widehat{h}^i(D) \neq 0$ if and only if there exists a subset $J \subseteq \Delta(1)$ such that $H_{|\Delta_J|}^i \neq 0$ and $\operatorname{Vol} P_{D,J} \neq 0$. Fix any such subset $J$ and let $I = \Delta(1) \setminus J$. 

In the proof of Corollary \ref{cones} we saw that $\operatorname{Vol} P_{D,J}>0$ if and only if there exists $D'$ linearly equivalent to $D$ satisfying the stated condition. So it remains to prove that $H_{|\Delta_J|}^i \iso \widetilde{H}^{i-1}(P_X^I)$.
 
For this, denote by $S$ the unit sphere in the vector space $N_{\mathbf R}$. Then there is an isomorphism
\begin{align*}
H_{|\Delta_J|}^i \iso \widetilde{H}^{i-1}(S \setminus S \cap |\Delta_J|).
\end{align*}
The fan of $X$ induces the structure of a simplicial complex on $S$, and the polytope $P_X$ can be viewed as the dual complex. Then $S \setminus S \cap |\Delta_J|$ retracts onto $P_X^I$. Combining with the previous displayed isomorphism, this gives an isomorphism
\begin{align*}
H_{|\Delta_J|}^i \iso \widetilde{H}^{i-1}(P_X^I).
\end{align*} 
\end{proof}

We emphasise that this gives a method for computing $q$-ample cones in practice. Given $X$, we take
its polytope $P_X$. For each $i$, we enumerate the subsets $I \subset
\Delta (1)$ such that $\widetilde{H}^i(P_X^I)$ is nontrivial. Each
such subset $I$ defines an orthant $O_I$ in $\mathbf R^{\Delta(1)}$
consisting of line bundles whose $i$-th asymptotic cohomology has a
nonzero contribution from $I$. Let $[O_I]$ denote the image of this
orthant in $N^1(X)$. Then the proposition shows that
$Amp_q(X)$ is the complement of $\bigcup_I \overline{[O_I]}$, where
the union is over all $I \subset \Delta (1)$ such that
$\widetilde{H}^i(P_X^I) \neq 0$ for some $i \geq q$. The next two subsections will illustrate this algorithm.

\subsection{Bundles over $\mathbf P^1$}
 We follow the notation for projective bundles from \cite[Chapter
   7]{CLS}. Let $X = \mathbf P(\mathcal V)$, where $\mathcal V$ is a vector bundle over
 $\mathbf P^1$ of rank $n+1$. By the Birkhoff theorem
 $\mathcal V$ is a direct sum of line bundles; after twisting, we can assume
 without loss of generality that $\mathcal V = \mathcal{O}_{\P^1} \oplus
 \mathcal{O}_{\P^1}(a_1) \oplus \cdots \oplus
 \mathcal{O}_{\P^1}(a_n)$, where $0 \leq a_1 \leq \cdots \leq a_n$.

The fan of $X$ is described as follows. Let $\R \times \R^n$ have basis $v_1, e_1, \ldots, e_n$, and set $$e_0=-\sum_{i=1}^n e_i \,; \quad v_0=-v_1+\sum_{i=1}^n a_i e_i.$$
Then the vectors $e_0,\ldots,e_n,v_0,v_1$ span the rays of the fan of $X$, and the top-dimensional cones are of the following form:
\begin{align*}
\left< v_i, e_0,\ldots,\widehat{e_j},\ldots,e_n \right> \quad i \in \{0,1\}, \, j \in \{0,\ldots,n\}.
\end{align*}
All of these cones are simplicial, so the codimension-1 cones in the fan, corresponding to the torus-invariant curves on $X$, are obtained by omitting one spanning vector from one of the cones above. Using the intersection formulas from \cite[Section 5.1]{Fulton} it is straightforward to calculate intersections between torus-invariant curves and divisors. This allows us to identify ample divisors on $X$: in particular, we find that the divisor 
\begin{align*}
A & := \sum_{i=0}^n E_i + \left( \sum_{i=1}^n a_i +1 \right) \left(V_0
+ V_1 \right)
\end{align*}
is ample. (Here the $E_i$ and $V_j$ are the torus-invariant divisors corresponding to the vectors $e_i$ and $v_j$ generating rays of the fan: geometrically, $E_i$ is the sub-bundle of $X$ obtained by quotienting $V$ by the summand $\mathcal{O}(a_i)$, and $V_j$ is the fibre over one of the torus-invariant points of $\P^1$.)

Given the ample divisor $A$, we have the corresponding polytope $P_X$ as described above:
\begin{align*}
P_X = \left\{ u \in M_\R : \left<u,e_i\right> \geq -1, \, \left< u, v_j \right> \geq -\sum a_i -1 \right\}. 
\end{align*}
The inequalities involving the $e_i$ cut out a polyhedron of the form $\R \times \Delta^n$;  the inequalities with $v_0$ and $v_1$ then bound this in the direction of the $\R$-factor. It is straightforward to check that the faces of the polytope corresponding to $v_0$ and $v_1$ are disjoint, irrespective of the values of the $a_i$, so that $P_X$ is combintorially equivalent to the polytope $[0,1] \times \Delta^n$. The homology groups of unions of faces of this polytope are described in the following lemma.

\begin{lemma} Label the the faces of the polytope $P_X$ so that the unique pair of disjoint faces are labelled $F_{n+2}$ and $F_{n+3}$. Let $Y$ be a union of closed top-dimensional faces of $P_X$. Then its reduced homology groups are
\begin{align*}
\widetilde{H}_k(Y) = 
\begin{cases} 
\kk & \text{ if } Y=\partial P_X, \, k=n  \\
\kk & \text{ if } Y = \partial P_X \setminus \left\{ F_{n+2} \cup
  F_{n+3} \right\}, \, k=n-1 \\
\kk & \text{ if } Y = F_{n+2} \cup F_{n+3}, \, k=0 \\
\kk & \text{ if } Y = \emptyset, \, k=-1 \\
0 & \text{ otherwise. }
\end{cases}
\end{align*}
\end{lemma}
\begin{proof} The proof comes  from the long exact sequence
of reduced homology groups associated to a sequence $A \hookrightarrow B
\arrow B/A$, where $A$ is a (reasonable) closed subspace of a
topological space $B$. Applying this with $B=Y$, a union of
top-dimensional faces of $P_X$, and $A=Y \cap F_{n+2}$ we reduce the problem to calculating the reduced homology of either a union of faces of a simplex or the disjoint union of a point with a union of faces of a simplex. Using the fact that the union of any proper subset of faces of a simplex has no reduced homology, the result follows.
\end{proof}

This lemma immediately identifies the index sets $I_\alpha$ which give
nonzero contributions to cohomology of a line bundle, as described above. We have
\begin{align*}
I_0 &= \left\{ \alpha \subseteq \Delta(1) =\left\{1,\ldots,n+3\right\} \, | \,
  \widetilde{H}^0(Z_\alpha,\kk) \neq 0 \right\} = \left\{ \left\{n+2,n+3
  \right\} \right\} \\
I_1 &= \cdots = I_{n-2} = \emptyset \\
I_{n-1} &=  \left\{ \alpha \subseteq \Delta(1) \, | \,
  \widetilde{H}^{n-1}(Z_\alpha,\kk) \neq 0 \right\} = \left\{ \left\{1,\ldots,n+1 \right\} \right\} \\
I_n &= \left\{ \alpha \subseteq \Delta(1) \, | \,
  \widetilde{H}^n(Z_\alpha,\kk) \neq 0 \right\} = \left\{
  \left\{1,\ldots,n+3 \right\}\right\}. \\
\end{align*}
The corresponding orthants in $\R^{\Delta(1)} \iso \R^{n+3}$ are then
\begin{align*}
  O_{I_0} &= \left\{ \left(d_1,\ldots,d_{n+3} \right) \in \R^{n+3} \, |
    \, d_{n+2} <0, d_{n+3} <0, d_i \geq 0 \text{ for all } i=1,\ldots,n+1 \right\} \\
  O_{I_1} &= \cdots = O_{I_{n-2}} = \emptyset \\
  O_{I_{n-1}} &= \left\{ \left(d_1,\ldots,d_{n+3} \right) \in \R^{n+3}
    \, | \, d_{n+2} \geq 0, \, d_{n+3} \geq 0, \, d_i<0 \text{ for all }
    i=1,\ldots,n+1
  \right\} \\
  O_{I_n} &= \left\{ \left(d_1,\ldots,d_{n+3} \right) \in \R^{n+3}
    \, |
    \, d_i<0 \text{ for all } i=1,\ldots,n+3 \right\} \\
\end{align*}

Now we can calculate the $q$-ample cones $Amp_q(X)$. Proposition
\ref{prop-orthants} says that $Amp_q(X)$ is the complement in $\Pic(X)$ of
the union of the images of all the closed orthants
$\overline{O_{I_i}}$ for $i \geq q$.  The map $\R^{n+3} \arrow
\Pic(X)$ has kernel equal to the column space of the matrix
\begin{align*}
\begin{pmatrix}
1 & 0 & \cdots & 0 & 0 \\
0 & 1 & \cdots & 0 & 0 \\
\vdots & \vdots & \ddots & \vdots & \vdots \\
-1 & -1 & \cdots & -1 & 0 \\
0 & 0 & \cdots & 0 & 1 \\
a_1 & a_2 & \cdots & a_n & -1
\end{pmatrix}
\end{align*}
whose rows are the primitive vectors of the rays of the fan of $X$, expressed in the basis $\left< e_1,\ldots,e_n,v_1\right>$. So in $\Pic(X)$ we have $E_i=E_0-a_iV_0$ and $V_0=V_1$. Let us denote the latter linear equivalence class by $V$. The images of the closed orthants $\overline{O_{I_i}}$ above are then

\begin{align*}
\overline{O_{I_0}} &\mapsto \left<E_0,E_0-a_1V,\ldots,E_0-a_nV,-V,-V\right> = \left<E_0,-V \right> \\
\overline{O_{I_{n-1}}} &\mapsto \left<-E_0,-E_0+a_1V,\ldots,-E_0+a_nV,V,V \right> = \left<-E_0,V \right> \\
\overline{O_{I_n}} &\mapsto \left<-E_0,-E_0+a_1V,\ldots,-E_0+a_nV,-V,-V \right> = \left<-E_0+a_nV,-V \right>
\end{align*}
where the last equality uses the fact that $a_i  \leq a_n$ for all $i$. 

Putting these regions together as described in Theorem \ref{thm-main}, we get the following result:

\begin{align*}
\overline{Amp_0(X)} &= \left< E_0, V \right> \\
\overline{Amp_q(X)} &= \left<V,-V\right>, \,\,\, 0<q<n\\
\overline{Amp_n(X)} &= \Pic(X)_\R \setminus \left<-E_0+a_nV,-V \right>.
\end{align*}

\subsection{$\P^1$-bundles over projective space}
In a similar way we can calculate the $q$-ample cones of a toric
$\P^1$-bundle over any projective space. Let $X=\P( \mathcal O_{\P^n}
\oplus \mathcal O_{\P^n} (a) )$.  Then the fan of $X$ is the
following: Let $\R^n \times \R$ have basis $v_1,\ldots,v_n,e_1$ and
set $$e_0=-e_1\, ; \ v_0 = -\sum_{i=1}^n v_i+ ae_1.$$ Then the maximal
cones in the fan are of the form
\begin{align*}
\left< v_0,\ldots,\widehat{v_i},\ldots, v_n,e_j \right> \quad \, i \in \{0,\ldots,n\}, \, j \in \{0,1\}. 
\end{align*}
Computing intersection numbers with torus-invariant curves shows that the divisor
$$ A:= \left(a+1 \right) (V_0+\cdots+V_n) + E_0 +E_1 $$ is ample; as
before, one finds that the polytope $P_A$ is combinatorially
equivalent to $\Delta^n \times [0,1]$. Repeating the process above, we
obtain the following result for the $q$-ample cones:

\begin{align*}
\overline{Amp_0(X)} &= \left< E_0, V \right> \\
\overline{Amp_q(X)} &= \left< V, -V \right> \\
\overline{Amp_n(X)} &= \Pic(X)_\R \setminus \left< -E_0+aV,-V \right>. 
\end{align*}
The orthants of cohomology nonvanishing and the $q$-ample cones in this example are shown in the figure below.

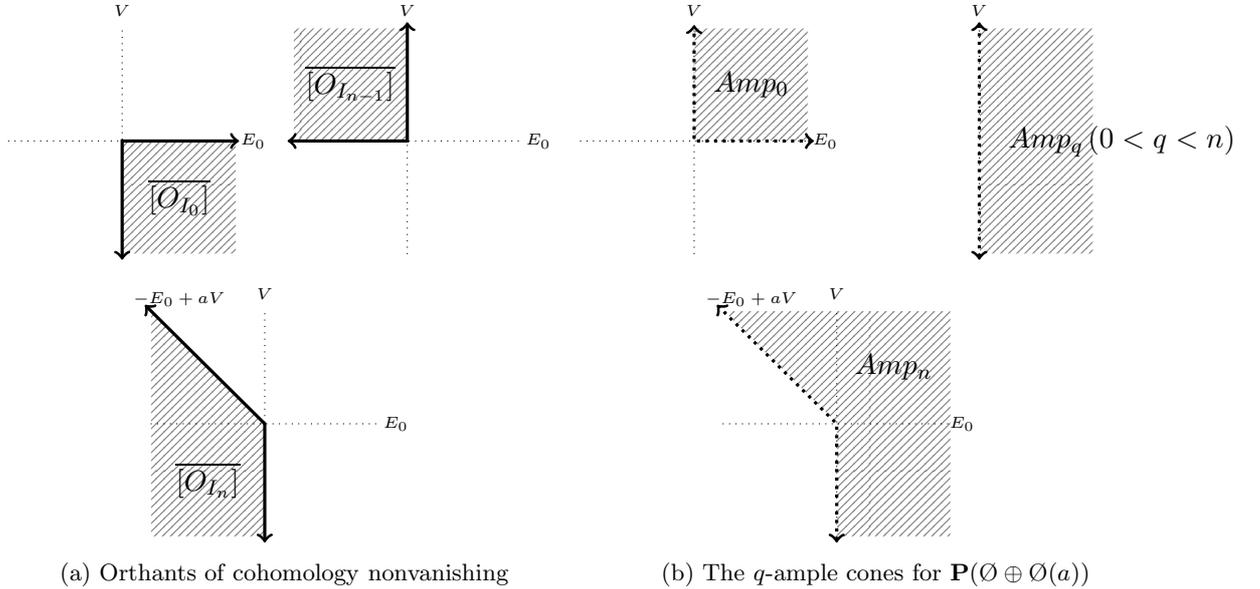
\begin{figure}[h]
\captionsetup[subfigure]{font=footnotesize}
\centering
\subcaptionbox{Orthants of cohomology nonvanishing}[.5\textwidth]{%
\begin{tikzpicture}[scale=0.75]
\draw[thin][draw=white][pattern=north east lines][pattern color=gray] (0,-2) -- (0,0) --
(2,0) -- (2,-2) --cycle;

\draw[dotted] (-2,0) -- (2,0) [shift={(0.3,0)}] node{\tiny{$E_0$}};
\draw[dotted] (0,-2) -- (0,2) [shift={(0,0.3)}] node{\tiny{$V$}};
\draw[<->][very thick] (0,-2.1) -- (0,0) -- (2.05,0);
\draw (1,-1) node{$\overline{[O_{I_0}]}$};

\draw [shift={(5,0)}][thin][draw=white][pattern=north east lines][pattern color=gray] (-2,0) -- (0,0) --
(0,2) -- (-2,2) --cycle;

\draw  [shift={(5,0)}][dotted] (-2,0) -- (2,0) [shift={(0.3,0)}] node{\tiny{$E_0$}};
\draw  [shift={(5,0)}][dotted] (0,-2) -- (0,2) [shift={(0,0.3)}] node{\tiny{$V$}};
\draw [shift={(5,0)}][<->][very thick] (-2.1,0) -- (0,0) -- (0,2.1);
\draw [shift={(5,0)}] (-1,1) node{$\overline{[O_{I_{n-1}}]}$};


\draw [shift={(2.5,-5)}][ thin][draw=white][pattern=north east lines][pattern color=gray] (-2,2) -- (0,0) --
(0,-2) -- (-2,-2) --cycle;

\draw [shift={(2.5,-5)}][<->][very thick] (0,-2.1) -- (0,0) -- (-2.1,2.1) [shift={(0.6,0.1)}] node{\tiny{$-E_0+aV$}};

\draw  [shift={(2.5,-5)}][dotted] (-2,0) -- (2,0) [shift={(0.3,0)}] node{\tiny{$E_0$}};
\draw  [shift={(2.5,-5)}][dotted] (0,-2) -- (0,2) [shift={(0,0.3)}] node{\tiny{$V$}};
\draw [shift={(2.5,-5)}] (-1,-1) node{$\overline{[O_{I_n}]}$};
\end{tikzpicture}}%
\subcaptionbox{The $q$-ample cones for $\P(\O \oplus \O(a))$}[.5\textwidth]{\begin{tikzpicture}[scale=0.75]
\draw[thin][draw=white][pattern=north east lines][pattern color=gray]
(2,0) -- (0,0) -- (0,2) --(2,2)--cycle;

\draw[dotted][<->][very thick] (2.1,0) -- (0,0) -- (0,2.05);
\draw[dotted] (-2,0) -- (2,0) [shift={(0.3,0)}] node{\tiny{$E_0$}};
\draw[dotted] (0,-2) -- (0,2) [shift={(0,0.3)}] node{\tiny{$V$}};

\draw (1,1) node{$Amp_0$};

\draw [shift={(5,0)}][thin][draw=white][pattern=north east lines][pattern color=gray] (0,2) -- (2,2) --
(2,-2)--(0,-2) --cycle;

\draw [shift={(5,0)}][<->][very thick][dotted] (0,2.1) -- (0,0) -- (0,-2.1);
\draw  [shift={(5,0)}][dotted] (0,-2) -- (0,2) [shift={(0,0.3)}] node{\tiny{$V$}};
\draw [shift={(5,0)}] (2.5,0) node{$Amp_q \, (0 < q < n)$};


\draw [shift={(2.5,-5)}][thin][draw=white][pattern=north east lines][pattern color=gray] (-2,2) -- (0,0) --
(0,-2) -- (2,-2) --(2,2) --cycle;

\draw [shift={(2.5,-5)}][<->][very thick][dotted] (0,-2.1) -- (0,0) -- (-2.1,2.1) [shift={(0.6,0.1)}] node{\tiny{$-E_0+aV$}};

\draw  [shift={(2.5,-5)}][dotted] (-2,0) -- (2,0) [shift={(0.2,0)}] node{\tiny{$E_0$}};
\draw  [shift={(2.5,-5)}][dotted] (0,-2) -- (0,2) [shift={(0,0.3)}] node{\tiny{$V$}};
\draw [shift={(2.5,-5)}] (1,1) node{$Amp_n$};
\end{tikzpicture}}
\end{figure}

~\newline
\noindent As a remark, it is well-known (\cite[General Properties
  1.5]{DPS}, \cite[Theorem 9.1]{Totaro}) that, for any projective
variety of dimension $d$, the $(d-1)$-ample cone is the complement in
$N^1(X)$ of the negative of the pseudoeffecitve cone. This gives an
easier way to calculate the cones $Amp_n$ in the examples
above. Similarly, the toric version of Kleiman's criterion gives the
ample cone $Amp_0$.

\section{Big $q-$ample line bundles} \label{restriction}

In this section we show that for big line bundles on toric varieties, $q$-ampleness can be detected by restriction to torus-invariant divisors. This can be deduced from a theorem of Brown \cite{Brown}, who proved that for big line bundles on arbitrary projective varieties, $q$-ampleness can be detected by restriction to the augmented base locus. We give the proof in the toric case here since it is simple and self-contained.

\begin{theorem}\label{toricrestriction}
Let $X$ be a simplicial toric variety and $L$ a big line bundle. Then $L$ is $q$-ample if and only if $L|_E$ is $q-$ample on each torus invariant divisor $E$.
\end{theorem}
\begin{proof}In one direction, if $L$ is $q$-ample, then so is its restriction to each subvariety of $X$ from the definition of $q$-ampleness.

For the other direction, by Proposition \ref{q-ample-prop} it is
enough to show that for a locally free sheaf $\E$ on $X$ there exists
$m_0>0$ such that for all $m \geq m_0$ and all $i > q$, the cohomology
groups $H^i(X,\E(mL))$ are zero.

Let $E_1,\ldots,E_r$ be the set of torus invariant divisors on $X$. If $L$ is big, then there is a positive integer $k$ such that $kL$ has a section of the form  $s=x_1^{m_r}\cdots x_r^{m_r}$, where $x_i$ is a section that defines $E_i$ and each $m_i$ is strictly positive. Let $D\in |kL|$ denote the divisor of $s$, supported on the union of the $E_i$:  by hypothesis, $L$ is $q$-ample on each $E_i$, and hence on $D$ and $D_{red}=E_1\cup \cdots \cup E_r$ by \cite[Proposition 2.3]{Ottem}. 

Now, if $\E$ is a locally free sheaf on $X$, we have the exact sequence
$$
0\to \E((m-1)D)\to \E(mD)\to \E(mD)\otimes \O_D\to 0
$$which, by the $q$-ampleness of $L$ on $D$, shows that $H^{i}(X,\E((m-1)L)\to H^{i}(X, \E(mL))$ is surjective for $i>q$ and $m$ large. It follows that there is an $m_0>0$ such that for each $i>q$, the canonical map
$$H^i(X,\E(mD))\to \varinjlim H^i(X,\E(mD))\simeq H^i(X-D,\E)
$$ is an isomorphism for each $m\ge m_0$. But the complement of $D$ is the torus $(\mathbb{C}^*)^{\text{dim} \, X}\subset X$. In particular, $X-D$ is an affine variety, and all higher cohomology groups vanish here. Hence $H^i(X,\E(m_0L))=0$ for each $i>q$ and $L$ is $q$-ample. 
\end{proof}



We  remark that the above proof applies unchanged to any
$\Q$-factorial Mori dream space $X$ if we let the $E_i$ denote any set
of divisors whose linear equivalence classes span the effective cone
of $X$. The fact that the complement of the union of the $E_i$ is
still affine follows from the fact that there is an embedding $X
\hookrightarrow T$ into a toric variety such that the $E_i$ are
exactly the restriction of the torus-invariant divisors of $T$: hence
$X-\cup_iE_i$ is the intersection of $X$ with a torus in $T$, and so
is affine.

\section{$q$-ampleness of $-K_X$} \label{flip}
A natural question is how to describe varieties for which $-K_X$ is $q$-ample, for different values of $q$. When $q=0$, this means $X$ is a Fano variety. When $q=\dim X-1$, as mentioned in Section \ref{examples}, this means that $K_X$ is not pseduoeffective, which in turn by Boucksom--Demailly--P\u{a}un--Peternell \cite{BDPP} means $X$ is uniruled. The geometric meaning of $q$-ampleness of $-K$ remains unclear for the intermediate cases $0<q<\dim X-1$. For instance, if $X$ is a threefold with $-K_X$ 1-ample, then $X$ need not be rationally connected: an example is  $X=\mathbb P(\Omega^1_S)$, where $S$ is a general quartic surface \cite[Example 5.6]{DPS}.

 For 3-folds, Demailly--Peternell--Schneider \cite[Problem 5.9]{DPS} asked whether 1-ampleness of $-K_X$ is preserved under flips. The following example gives a negative answer to this question. 

We consider two projective toric varieties $X=X(\Delta_1)$ and $Y=Y(\Delta_2)$ whose fans have rays spanned by the columns of the matrix
$$\left(
\begin{array}{cccccc}
 1 & 0 & 0 & 2 & 1 & -1 \\
 0 & 1 & 0 & -1 & 0 & 0 \\
 0 & 0 & 1 & 1 & -1 & 0
\end{array}
\right)$$and whose maximal cones are the following:
\begin{eqnarray*}\Delta_1&:&\{\langle 0,1,2 \rangle ,\langle 0,2,3\rangle ,\langle 0,3,4\rangle , \langle 0,4,1\rangle ,\langle 5,1,2\rangle , \langle 5,2,3\rangle , \langle5,3,4\rangle , \langle5,4,1\rangle \}\\
\Delta_2&:&\{\langle 0,1,3\rangle ,\langle 1,2,3 \rangle ,\langle 0,3,4 \rangle,\langle 0,4,1\rangle ,\langle 5,1,2\rangle ,\langle 5,2,3\rangle ,\langle 5,3,4\rangle ,\langle 5,4,1\rangle \}
\end{eqnarray*}
(Here an integer $i$ denotes the ray spanned by the $i$-th column of the matrix.) 

The variety $X$ is smooth, while $Y$ has exactly one singular point, which is a $\Z/2$-quotient singularity. The fans of $X$ and $Y$ are both refinements of the fan $\Delta_3$ obtained by replacing the first two cones in either fan above by the non-simplical cone $\langle 0,1,2,3 \rangle$. If $Z$ is the toric variety defined by $\Delta_3$, then both $X$ and $Y$ are partial resolutions of $Z$; in particular there is a birational map $\phi: X \dashrightarrow Y$. The indeterminacy locus of $\phi$ is the rational curve $C$ corresponding to the cone $\langle 0,2 \rangle$. One calculates that $K_X \cdot C <0$, and so $\phi: X \dashrightarrow Y$ is a flip.

\begin{proposition}
Let $X$ and $Y$ be as above. Then $-K_X$ is 1-ample, while $-K_Y$ is not.
\end{proposition}

\begin{proof}
On any simplicial toric variety, the anticanonical divisor is big, so we can apply Theorem \ref{toricrestriction} to the divisors $-K_X$ and $-K_Y$. One checks easily that $-K_X$ restricts to a 1-ample line bundle on each torus-invariant surface in $X$. (Recall that the 1-ample cone of a surface is the complement of the negative of the pseudoeffective cone.) On the other hand, it is straightforward to check using the formulas of \cite[Section 5.1]{Fulton} that the ($\mathbf Q$-Cartier) divisor $-K_Y$ is numerically trivial when restricted to the divisor $D_0$ corresponding to the vector $(1,0,0)$, so it cannot be 1-ample. 
\end{proof}

\section{A Kodaira-type vanishing theorem} \label{kodaira}
One reason for studying partial positivity comes from the possibility of vanishing theorems. Unfortunately, the analogue of the Kodaira vanishing theorem does not hold for $q$-ample line bundles in general: in fact, it fails already in the case of the 3-dimensional flag variety $SL_3/B$. In this section, we will show that the Kodaira vanishing does hold on a projective toric variety.

We mention that Greb--K\"uronya \cite{GK} proved a related vanishing theorem for $q$-ample line bundles under the additional assumption that the line bundle admits a global section with mild singularities (so that the usual proof of Kodaira's vanishing theorem using Hodge theory goes through for higher $q$). In our case, however, we do not require the line bundle to be effective.

The main ingredient of the proof is \emph{multiplication maps}. Let $X=X(\Delta)$ be the toric variety defined by a fan $\Delta$ in a
lattice $N$. For each integer $m \geq 0$ we have the
multiplication-by-$m$ (or ``toric Frobenius") map $F_m: N \arrow N$. This induces a finite
surjective toric morphism $f:X \rightarrow X$ with the property that
$\mathcal O_X \arrow f_* \mathcal O_X$ splits \cite[Proposition 3.1]{Payne}. If $L$ is any line bundle on $X$, this gives a split injection 
\begin{align} \label{isos}
L \arrow L \otimes f_* \mathcal O_X &\iso f_*(f^* L) \iso f_*(L^m).
\end{align}
The existence of this split injection quickly yields the following:
\begin{theorem}[Kodaira-type vanishing]
Let $X$ be a projective Cohen--Macaulay toric variety, and let $L$ be a $q$-ample line bundle on $X$. Then 
\begin{align*}
H^i(X,L^{-1})=0 \text{ for } i<n-q.
\end{align*}
\end{theorem}
\begin{proof} Let $F_m$ be the multiplication-by-$m$ map and $f :X \rightarrow X$ the corresponding map on $X$.
 Since cohomology commutes with direct sums, by (\ref{isos}) we get, for each $i \geq 0$, a split injection
\begin{align} \label{inj}
H^i(X,L^{-1}) \arrow & H^i(X,f_*(L^{-m})) \iso  H^i(X,L^{-m})
\end{align}
where the last isomorphism comes from the Leray spectral sequence and
finiteness of $f$. 

Now let $\omega_X^\circ$ be the dualising sheaf of $X$
\cite[Proposition 7.5]{Hartshorne}. Then by Serre duality the
rightmost group in \eqref{inj} is dual to $H^{n-i}(X,L^m \otimes
\omega_X^\circ)$, so by $q$-ampleness of $L$ it vanishes if $n-i>q$
and $m$ is sufficiently large.
\end{proof}

\small
{\sc Insitut f\"ur Algebraische Geometrie, Leibniz Universit\"at Hannover, Welfengarten 1, Hannover 30167, Germany.} {\tt broomhead@math.uni-hannover.de}

{\sc DPMMS, University of Cambridge, Wilberforce Road, Cambridge CB3 0WB, United Kingdom} {\tt J.C.Ottem@dpmms.cam.ac.uk}

{\sc Department of Mathematical Sciences, Loughborough University, LE11 3TU, United Kingdom.} {\tt A.Prendergast-Smith@lboro.ac.uk
\end{document}